\documentclass{article}
\usepackage{amsmath}
\usepackage{amssymb,comment,amscd}
\usepackage{mathrsfs,wasysym}
\usepackage{bm}
\usepackage{euscript,color}

\usepackage{libertine}
\usepackage[T1]{fontenc}

\usepackage[utf8]{inputenc}

\usepackage{amsmath}
\usepackage{amssymb}
\usepackage{amsthm}
\usepackage{mathrsfs}

\usepackage{graphicx}
\graphicspath{ {/desktop/graphs/} }

\usepackage{version}
\usepackage[margin=3cm]{geometry}
\usepackage{color}
\usepackage{mathtools}
\usepackage{tikz}
\usetikzlibrary{arrows}
\usetikzlibrary{trees}
\usepackage{float}
\usepackage{textcomp}
\usepackage{bbm}

\numberwithin{equation}{section}

\newtheorem{claim*}{Claim}[section]

\newtheorem{theorem}{Theorem}[section]
\newtheorem{corollary}{Corollary}[section]

\newtheorem{remark}{Remark}
\newcommand\numberthis{\addtocounter{equation}{1}\tag{\theequation}}

\DeclareMathOperator{\E}{\mathbb{E}}
\DeclareMathOperator{\J}{\mathcal{J}}

\DeclareMathOperator{\Q}{\mathcal{Q}}
\DeclareMathOperator{\Pb}{\mathbb{Pb}}

\begin{document}

\title{Occupation times for the finite buffer fluid queue \\with phase-type ON-times}
\author{N. J. Starreveld, R. Bekker and M. Mandjes}
\maketitle

\begin{abstract}
\noindent    
In this short communication we study a fluid queue with a finite buffer. The performance measure we are interested in is the occupation time over a finite time period, i.e., the fraction of time the workload process is below some fixed target level. Using an alternating renewal sequence, we determine the double transform of the occupation time; the occupation time for the finite buffer M/G/1 queue with phase-type jumps follows as a limiting case.


\paragraph{Affiliations:} {N. J. Starreveld} is with Korteweg-de Vries Institute for Mathematics, Science Park 904,
1098 XH Amsterdam, University of Amsterdam, the Netherlands. Email: {\footnotesize {\tt n.j.starreveld@uva.nl}}.
{R. Bekker} is with Department of Mathematics, Vrije Universiteit Amsterdam, De Boelelaan 1081a, 1081
HV Amsterdam, The Netherlands. Email: {\footnotesize {\tt r.bekker@vu.nl}}.
{M. Mandjes} is with Korteweg-de Vries Institute for Mathematics, University of Amsterdam, Science
Park 904, 1098 XH Amsterdam, the Netherlands; he is also affiliated with EURANDOM, Eindhoven
University of Technology, Eindhoven, the Netherlands, and CWI, Amsterdam the Netherlands.
Email: {\footnotesize {\tt m.r.h.mandjes@uva.nl}}.


\end{abstract}


\section{Introduction}

Owing to their tractability, the OR literature predominantly focuses  on queueing systems with an {\it infinite} buffer or storage capacity. In
practical applications, however, we typically encounter systems with  {\it finite}-buffer queues. Often, the infinite-buffer queue is used to  approximate its finite-buffer counterpart, but it 
is questionable whether this is justified when the buffer is not so large. 

In specific cases explicit analysis of the finite-buffer queue {\it is}\, possible.  In this paper we consider the workload process  $\{Q(t)\}_{t\geq0}$ of a fluid queue with finite workload capacity $K>0$. Using the results for the fluid queue we also analyze the finite-buffer M/G/1 queue.
The performance measure we are interested in is the so-called {\it occupation time} of the set $[0,\tau]$ up to time $t$, for some $\tau \in [0,K]$, defined by 
\begin{equation}
\label{eq: Occupation time}
\alpha(t) := \int_0^t 1_{\{Q(s)\in [0,\tau]\}}{\rm d}s.
\end{equation}
Our interest in the occupation time can be motivated as follows.
The queueing literature mostly focuses on stationary performance measures (e.g. the distribution of the workload $Q(t)$ when $t\to\infty$) or on the performance after a finite time (e.g. the distribution of $Q(t)$ at a fixed time $t\geq0$). Such metrics do not always provide operators with the right means to assess the service level agreed upon with their clients. Consider for instance a call center in which the service level is measured over intervals of several hours during the day; a typical service-level target is then that 80\% of the calls should be answered within 20 seconds. Numerical results for this call center setting \cite{Roubos, Roubos2} show that there is severe fluctuation in the service level, even when measured over periods of several hours up to a day. Using a stationary measure for the average performance over a finite period may thus be highly inadequate (unless the period over which is averaged is long enough). The fact that the service level fluctuates on such rather long time scales has been observed in the queueing community only relatively recently (see \cite{Baron, Roubos, Roubos2} for some call center and queueing applications). Our work is among the first attempts to study occupation times in finite-capacity queueing systems. 

Whereas there is little literature on occupation times for queues, there is a substantial body of work  on occupation times in a broader setting. One stream of research focuses on occupation times for processes whose paths can be decomposed into regenerative cycles \cite{Cohen, Star, Tak, Zac}. Another branch is concerned with occupation times of spectrally negative L\'evy processes, see e.g. \cite{Lan, Loe, KypRef}. The results established typically concern occupation times until a first passage time, whereas \cite{KypRef} focuses on refracted L\'evy processes. In \cite{Star} spectrally positive L\'evy processes with reflection at the infimum were studied as a special case; we also refer to \cite{Star} and references therein for additional literature. A natural extension of L\'evy processes are Markov-modulated L\'evy processes; for the case of a Markov-modulated Brownian motion the occupation time has been analyzed in \cite{Bre}. To the best of our knowledge there is no paper on occupation times for doubly reflected processes, as we consider here.

In this paper we use the framework studied in \cite{Star}. More specifically, the occupation time is cast in terms of an alternating renewal process, whereas for the current setting the upper reflecting barrier complicates the analysis.
We consider a finite buffer fluid queue where during ON times the process increases linearly and during OFF times the process decreases linearly. We consider the case that ON times have a phase-type distribution and the OFF times have an exponential distribution. This framework allows us to exploit the regenerative structure of the workload process and provides the finite-capacity M/G/1 queue with phase-type jumps as a limiting special case. For this model we succeed in deriving closed-form results for the  Laplace transform (with respect to $t$) of the occupation time. Relying on the ideas developed in \cite{Asm2}, all quantities of interest can be explicitly computed as solutions of systems of linear equations. 

The structure of the paper is as follows. 
In Section~2 we describe the model and give some preliminaries. Our results are presented in Section~3. A numerical implementation of our method is presented in Section \ref{sec: numerics}.


\section{Model description and preliminaries}
\label{sec: Model}

We consider the finite capacity fluid queue with linear rates. The rate is determined by an independently evolving Markov chain, where we assume that there is only one state in which work decreases; this may be interpreted as the OFF time of a source that feeds work into the queue. There are multiple states of the underlying Markov chain during which work accumulates at (possibly) different linear rates. In case these rates are identical, periods during which work accumulates may be interpreted as ON times of a corresponding ON-OFF source. The ON-OFF source then has exponentially distributed OFF times, whereas the ON times follow a phase-type distribution.
The workload capacity is $K$ and work that does not fit is lost; see Subsection~\ref{subs:double-refl} for a more formal description. Some basic results concerning phase-type distributions and martingales that are used in the sequel are first presented in Subsection~\ref{subs:pre}.

\subsection{Preliminaries} \label{subs:pre}

\paragraph{Phase-type distributions}

 A phase-type distribution $B$ is defined as the {\em absorption time} of a continuous-time Markov process $\{\mathcal{J}(t)\}_{t\geq0}$ with finite state space $E\cup \{\partial\}$ such that $\partial$ is an absorbing state and the states in $E$ are transient. We denote by ${\bm\alpha}_0$ the initial probability distribution of the Markov process, by ${\bf T}$ the {\em phase generator}, i.e., the $|E|\times |E|$ rate matrix between the transient states and by ${\bf t}$ the {\em exit vector}, i.e., the $|E|$- dimensional rate vector between the transient states and the absorbing state $\partial$. The vector ${\bf t}$ can be equivalently written as $-\textbf{T}{\bf 1}$, where ${ \bf 1}$ is a column vector of ones. We denote such a phase-type distribution by $(n, {\bm \alpha}_0, {\bf T})$ where $|E| = n$. The cardinality of the state space $E$, i.e., $n$, represents the {\em number of phases} of the phase-type distribution $B$; for simplicity we assume that $E=\{1,\hdots,n\}$. In what follows we denote by $B$ a phase-type distribution with representation $(n, {\bm \alpha}_0, {\bf T})$; for a phase-type distribution with representation $(n,{\bf e}_i, {\bf T})$ we add the subscript $i$ in the notation. An important property of the class of phase-type distributions is that it is dense (in the sense of weak convergence) in the set of all probability distributions on $(0,\infty)$; see \cite[Thm. 4.2]{Asm}. For a phase-type distribution with representation $(n, {\bm \alpha}_0, {\bf T})$, the cumulative distribution function $B(\cdot)$, the density $b(\cdot)$ and the Laplace transform $\hat{B}[\cdot]$ are given in \cite[Prop. 4.1]{Asm}.  In particular, for $x\geq0$ and $s\geq0$, we have 
\begin{equation} 
\label{formulas}
\Pb(B> x) = -{\bm \alpha}_0^{\hspace{2pt} {\rm T}} e^{{\bf T}x}\textbf{1} \hspace{5mm} \text{and} \hspace{5mm} \hat{B}[s] = {\bm \alpha}_0^{\hspace{2pt} {\rm T}} (s{\bf I}-{\bf T})^{-1}{\bf t}.
\end{equation}
When the phase-type distribution has representation $(n,{\bf e}_i,{\bf T})$ we use the notation $\hat{B}_i(\cdot)$ instead of $\hat{B}(\cdot)$. For a general overview of the theory of phase-type distributions we refer to \cite{Asm, Asm2} and references therein.

\paragraph{Markov-additive fluid process (MAFP)}

Markov-additive fluid processes belong to a more general class of processes called {\em Markov-additive processes}, see \cite[Ch. XI]{Asm}. Consider a right-continuous irreducible Markov process $\{\mathcal{J}(t)\}_{t\geq0}$ defined on a filtered probability space $(\Omega, \mathcal{F}, \Pb)$ with a finite state space $E=\{1,\hdots,n\}$ and rate transition matrix $\mathcal{Q}$. While the Markov process $\mathcal{J}(\cdot)$ is in state $i$ the process $X(\cdot)$ behaves like a linear drift $r_i$. We assume that the rates $r_1,\hdots,r_n$ are independent of the process $\mathcal{J}(\cdot)$. Letting $\{T_i, i\geq0\}$ be the jump epochs of the Markov process $\mathcal{J}(\cdot)$ (with $T_0=0$) we obtain the following expression for the process $X(\cdot)$, 
\begin{align*}
X(t) &= X_0 +\sum_{m\geq1}\sum_{\substack {1\leq i\leq n }}\Bigg( r_i\left(T_m-T_{m-1}\right)1_{\{\J(T_{m-1})=i, T_m\leq t\}}+r_i (t-T_{m-1})1_{\{\J(T_{m-1})=i, T_{m-1}\leq t<T_m\}}\Bigg), \numberthis{\label{eq:MAPl}}
\end{align*}
where $t\geq0$ and $X_0\in\mathcal{F}_0$ is independent of the Markov process $\J(\cdot)$ and the rates $r_1,\hdots,r_n$. The process $X(\cdot)$ defined in (\ref{eq:MAPl}) will be referred to as a {\em Markov-additive fluid process} and abbreviated as MAFP. For $z\in\mathbb{C}^{\text{Re}\geq0}$, the {\em matrix exponent} of the MAFP is defined as 
\begin{equation}
\label{MAPexp}
F(z) = \mathcal{Q} -z\hspace{1pt}\text{diag}(r_1,\hdots,r_n)=\Q -z\hspace{1pt}{\bf \Delta}_r.
\end{equation} 
In what follows we shall need information concerning the roots of the equation 
\begin{equation}
\text{det}\left(F(z) - q\textbf{I}\right) = \text{det}(\Q-z{\bf \Delta}_r - q\textbf{I})=0,
\end{equation}
where ${\bf \Delta}_r=\text{diag}(r_1,\hdots,r_n)$ and $q\geq0$. From \cite{Iva} we have that there exist $n$ values $\rho_1(q),\hdots,\rho_n(q)$ and corresponding vectors ${\bf h}_1(q),\hdots,{\bf h}_n(q)$ such that, for each $k=1,\hdots,n$, $\text{det}\left(\Q-\rho_k(q){\bf \Delta}_r - q\textbf{I}\right)=0$ and $(\Q-\rho_k(q){\bf \Delta}_r - q\textbf{I}){\bf h}_k(q)=0$. 

\paragraph{The Kella-Whitt martingale} The counterpart of the Kella-Whitt martingale for {\em Markov-additive processes} was established in \cite{AsmK}; let $\{Y(t)\}_{t\geq0}$ be an adapted continuous process having finite variation on compact intervals. Set $Z(t) = X(t) + Y(t)$ and let $z\in\mathbb{C}^{\text{Re}\geq0}$. Then, for every initial distribution $(X(0),\J(0))$,
\begin{equation}
\label{KWM}
{\bf M}(z,t): = \int_0^t e ^{-z Z(s)}{\bf e}_{\J(s)}{\rm d}s F(z) + e^{-z Z(0)}{\bf e}_{\J(0)} - e^{-z Z(t)}{\bf e}_{\J(t)} - z\int_0^t e^{-z Z(s)}{\bf e}_{\J(s)} {\rm d}Y(s)
\end{equation}
is a vector-valued zero mean martingale.

\subsection{Fluid model with two reflecting barriers} \label{subs:double-refl}

 The MAFP $(X(t),\J(t))_{t\geq0}$ we analyze has a modulating Markov process $\{\J(t)\}_{t\geq0}$ with state space $E=\{1,\hdots, n+1\}$ and generator $\Q$ given by  
\begin{equation}
\label{generator}
\Q = \begin{bmatrix}
       -\lambda &  \lambda{\bm \alpha}_0^{{\rm T}}\\[0.5em]
       {\bf t}            & {\bf T}
  
     \end{bmatrix},
 \end{equation}
which is a $(n+1)\times (n+1)$ matrix. Additionally we suppose that $\lambda>0$, ${\bf t}$ is a $n\times1$ column vector with non-negative entries, ${\bm \alpha}_0$ is a $n\times 1$ column vector with entries that sum up to one and ${\bf T}$ is a $n\times n$ matrix with non-negative off-diagonal entries. The column vector ${\bf t}$ and the matrix ${\bf T}$ are such that each row of $\Q$ sums up to one, alternatively we can write ${\bf t} = -{\bf T}\textbf{1}$. On the event ${\{\J(\cdot)=1\}}$ the process $X(\cdot)$ decreases linearly with rate $r_1<0$ and on the event $\{\J(\cdot)=i\}$, for $i=2,\hdots,n+1$, $X(\cdot)$ increases linearly with rate $r_i>0$. Such a MAFP decreases linearly with rate $r_1$ during OFF-times, which are exponentially distributed with parameter $\lambda$, and increases linearly with rates $r_i$ during ON-times, which have a phase-type $(n,{\bm \alpha}_0,{\bf T})$ distribution. Depending on the state of the modulating process we have a different rate.  This model is motivated by finite capacity systems with an alternating source: during OFF times work is being served with rate $r_1$ while during ON times work accumulates with rates $r_2,\hdots,r_{n+1}$. 

The workload process $\{Q(t)\}_{t\geq0}$ we are interested in is formally defined as a solution to a two sided Skorokhod problem, i.e., for a Markov-additive fluid process $\{X(t)\}_{t\geq0}$ as defined in (\ref{eq:MAPl}), we have
\begin{equation}
\label{Skorokhod}
Q(t) = Q(0)+X(t) +L(t) - \bar{L}(t).
\end{equation}
In the above expression $\{L(t)\}_{t\geq0}$ represents the local time at the infimum and $\{\bar{L}(t)\}_{t\geq0}$ the local time at $K$. Informally, for $t>0$, $L(t)$ is the amount that has to be added to $X(t)$ so that it stays non-negative while $\bar{L}(t)$ is the amount that has to be subtracted from $X(t)+L(t)$ so that it stays below level $K$. It is known that such a triplet exists and is unique, see \cite{Kella, Kruk}. For more details we refer to \cite{Man} and references therein. For notational simplicity we assume that $Q(0) =\tau$ and that $\J(0)=1$, i.e., we start with an OFF time; the cases $\{Q(0)<\tau, \J(0)\neq1\}$ and $\{Q(0)>\tau, \J(0)\neq1\}$ can be dealt with analogously at the expense of more complicated expressions.
For the MAFP described above the matrix exponent is a $(n+1)\times(n+1)$ matrix. For $q>0$, denote by $\rho_1(q),\hdots,\rho_{n+1}(q)$ the $n+1$ roots of the equation $\text{det}(\Q-z{\bf \Delta}_r-q\textbf{I})=0$ and consider, for $k=1,\hdots,n+1$, the vectors ${\bf h}_k(q)= ({\rm h}_{k,1}(q),\hdots,{\rm h}_{k,n+1}(q))$ defined by 
\begin{equation}
\label{vectors}
{\rm h}_{k,1}(q) =1 \hspace{2mm} \forall k=1,\hdots,n+1 \hspace{5mm} \text{and} \hspace{5mm} {\rm h}_{k,j}(q)= -{\bf e}_{j-1}^{\hspace{2pt} {\rm T}} ({\bf T}-\rho_k(q)\bar{{\bf \Delta}}_r - q\textbf{I})^{-1}{\bf t} \hspace{2mm}\text{for}\hspace{2mm} j=2,\hdots,n+1,
\end{equation}
where ${\bf e}_j$ is the $n\times1$ unit column vector with 1 at position $j$, ${\bf T}$ and ${\bf t}$ are as in (\ref{generator}) and $\bar{\bf \Delta}_{r}$ is the $n\times n$ diagonal submatrix of $\bar{\bf \Delta}_{r}$ with $r_{j}$ at position $(j-1,j-1)$, for $j=2,\hdots,n+1$.
For the vectors defined in (\ref{vectors}) we have that $(\Q-\rho_k(q){\bf \Delta}_r-q\textbf{I}){\bf h}_k(q) =0$ for all $k=1,\hdots,n+1$.


\section{Result}

\subsection{The Markov Additive Fluid Process}
\label{sec: MAFP}

For the analysis of the occupation time $\alpha(\cdot)$ we observe that the workload process $\{Q(t)\}_{t\geq0}$ alternates between the two sets $[0,\tau]$ and $(\tau,K]$. Due to the definition of $\{X(t)\}_{t\geq0}$ both upcrossings and downcrossings of level $\tau$ occur with equality. Moreover, we see that an upcrossing of level $\tau$ can occur only when the modulating Markov process is in one of the states $2,\hdots,n+1$. Similarly, a downcrossing of level $\tau$ can occur only while the modulating Markov process is in state 1. We define the following first passage times, for $\tau\geq0$,
\begin{equation*}
\sigma:=\inf_{t>0}\{t:Q(t)= \tau\,|\, Q(0) = \tau, \J(0)=1\},
\:\:\:
T := \inf_{t>0}\{t: Q(t) = \tau\,|\, Q(0) = \tau, \J(0)\neq1\}.
\end{equation*}   
We use the notation $(T_i)_{i\in{\mathbb N}}$ for the sequence of successive downcrossings and $(\sigma_i)_{i\in{\mathbb N}}$ for the sequence of successive upcrossings of level $\tau$. An extension of \cite[Thms.\ 1 and 2]{Cohen} for the case of doubly reflected processes shows that $(T_i)_{i\in{\mathbb N}}$ is a renewal process, and hence the successive sojourn times, $D_1 := \sigma_1$, $D_i:=\sigma_i - T_{i-1}$, for $i\geq2$, and $U_i := T_i - \sigma_i$, for $i\geq1$, are sequences of well defined random variables. In addition, $D_{i+1}$ is independent of $U_i$ while in general $D_i$ and $U_i$ are dependent. We observe that the random vectors $(D_i,U_i)_{i\in\mathbb{N}}$ are i.i.d.\ and distributed as a generic random vector $(D,U)$. In Figure\ \ref{Figure PathG} a realization of $Q(\cdot)$ is depicted.

\begin{figure}[H]
   \centering
    \includegraphics{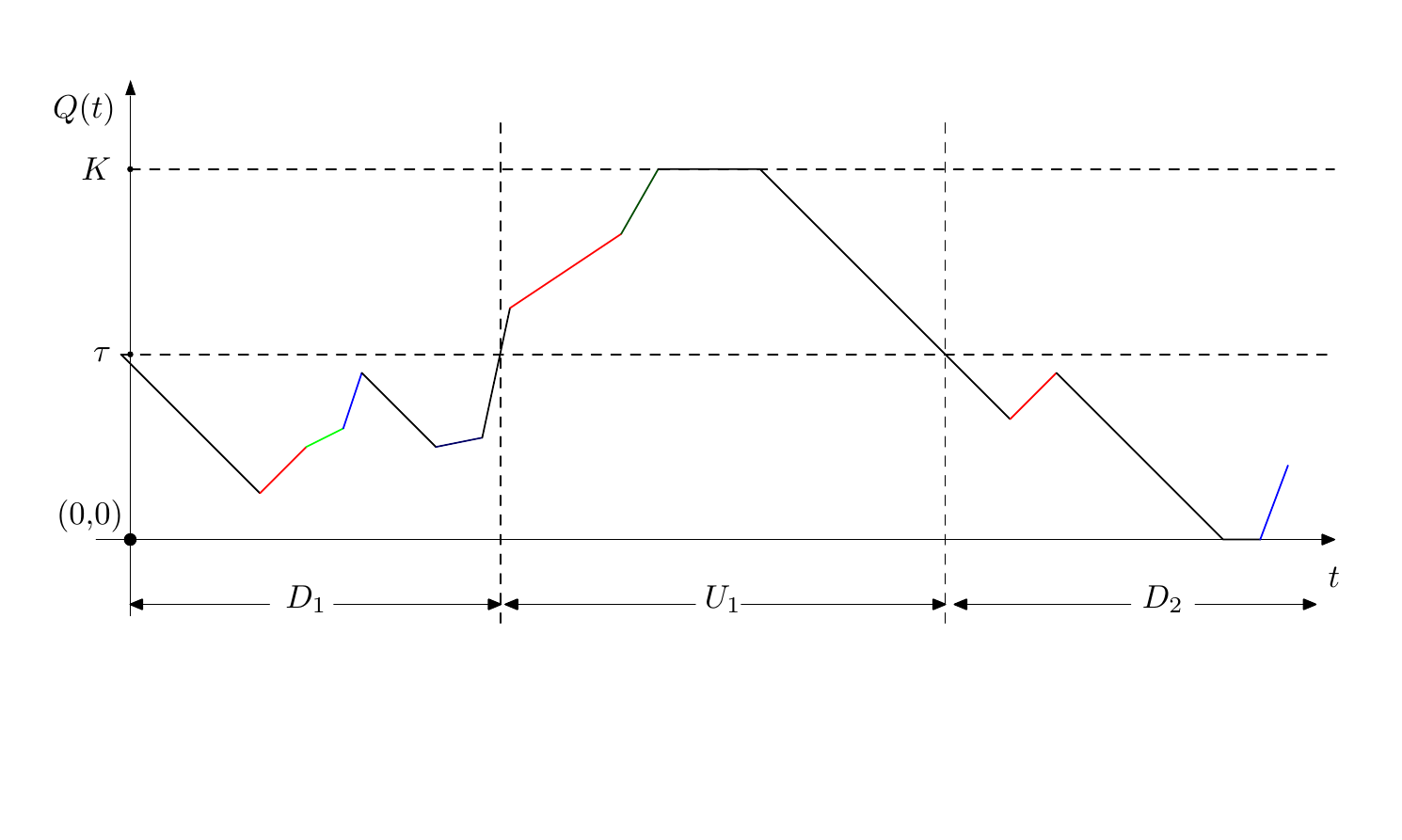}
\vspace{-70pt}
\caption{The workload process in a finite capacity fluid queue}
    \label{Figure PathG}
	
\end{figure}

\vspace{20pt}

The double transform of the occupation time $\alpha(\cdot)$ in terms of the joint transform of $D$ and $U$ is given in \cite[Theorem 3.1]{Star} which we now restate:

\begin{theorem}
\label{Old}
For the transform of the occupation time $\alpha(\cdot)$, and for $q\geq0, \theta\geq0$, we have 
\begin{equation}
\label{DoubleTransform}
 \int_0^\infty e^{-qt}\E e^{-\theta \alpha(t)}{\rm d}t =\frac{1}{1-L_{1,2}(q+\theta,q)}\Bigg[\frac{1-L_1(q+\theta)}{q+\theta}+\frac{L_1(q+\theta)-L_{1,2}(q+\theta,q)}{q}\Bigg],
\end{equation}

\hspace{-17pt} where, for $\theta_1,\theta_2\geq0$,
\[L_{1,2}(\theta_1,\theta_2) = \E e^{-\theta_1 D-\theta_2 U} \hspace{5mm} \text{and} \hspace{5mm} L_1(\theta_1) =L_{1,2}(\theta_1,0)= \E e^{-\theta_1 D}.\]
\end{theorem}
\begin{remark} 
Theorem \ref{Old} holds when the time instances $D_i+U_i, i=1,2,\hdots$ are regenerative points, i.e. the pairs $(D_i,U_i)_{i\geq1}$ form a sequence of i.i.d. random vectos. This property does not hold in general for an arbitrary MAFP, it holds though for the MAFP presented in Section \ref{subs:double-refl} as there is only one state with negative drift.

\end{remark}

To analyze the occupation time it thus suffices to determine the joint transform of the random variables $D$ and $U$, i.e., $L_{1,2}(\cdot,\cdot)$. The Laplace transform of the random variables $D$ and $U$ have been derived in \cite{Bean} for MAFPs. Our proofs can be trivially extended in order to derive the joint transform of the random variables $D$ and $U$ for a general MAFP, i.e. with multiple states with negative drifts, which is an additional novelty of this paper. Our interest lies though in the occupation time $\alpha(\cdot)$ so we restrict ourselves to the MAFP considered in Section \ref{subs:double-refl}. 

 For $i=2,\hdots,n+1$ we define the first hitting time of level $\tau$ with initial condition $(X(0),\J(0))=(\tau,i)$ as follows:
\[T_i :=\inf_{t\geq0}\{t:Q(t)=\tau|Q(0)=\tau, \J(0) = i \} \hspace{3mm} \text{and} \hspace{3mm} w_i(\theta_2) = \E \left[e^{-\theta_2 T_i}\right]=\E\left[ e^{-\theta_2 T}|\J(0)=i\right].\]
Considering the event $E_i$ that an upcrossing of level $\tau$ occurs while the modulating process $\J(\cdot)$ is in state $i$, for $i=2,\hdots,n+1$, we obtain, for $\theta_1,\theta_2\geq0$, 
\begin{equation}
\label{eq: DEC}
\E \left[e^{-\theta_1 D-\theta_2 U}\right] = \E\left[ e^{-\theta_1\sigma-\theta_2 T}\right]=\sum_{i=2}^{n+1} \E\left[ e^{-\theta_1\sigma}1_{\{E_i\}}\right]\E\left[ e^{-\theta_2 T}|E_i\right]=\sum_{i=2}^{n+1} \E\left[ e^{-\theta_1\sigma}1_{\{E_i\}}\right]\E \left[e^{-\theta_2 T_i}\right].
\end{equation}
In what follows we use, for $\theta_1\geq0$ and $i=2,\hdots,n+1$, the notation 
\begin{equation}
\label{eq:zi}
z_i(\theta_1):=\E\left[e^{-\theta_1\sigma}1_{\{E_i\}}\right] = \E\left[ e^{-\theta_1 \sigma}1_{\{\J(\sigma)=i\}}\right].
\end{equation}
It will be shown that these terms can be computed as the solution of a system of linear equations.
The idea of conditioning on the phase when an {\em upcrossing} occurs and using the {\em conditional independence} of the corresponding time epochs has been developed in \cite{Asm2}. The factors appearing in each term in (\ref{eq:zi}) can also be determined using the results of \cite{Bean}. Determining the factors involved in the terms presented above is the main contribution of the analysis that follows. We first present the exact expression for the double transform of the random variables $(D,U)$.

\begin{theorem}
\label{th:Jointtransform}
For $\theta_1,\theta_2\geq0$, the joint transform of the random variables $D$ and $U$ is given by 
\begin{equation}
\label{jointtransform}
\E e^{-\theta_1D-\theta_2U} = \frac{1}{C(\theta_2)}\sum_{i=2}^{n+1} z_i (\theta_1)\sum_{k=1}^{n+1}(-1)^{k+1}c_k(\theta_2){\rm h}_{k,i}(\theta_2),
\end{equation}
where the quantities $c_k(\theta_2), k=1,\hdots,n+1$ and $C(\theta_2)$ depend only on $\theta_2$ and are defined below in (\ref{DET}) and (\ref{eq:CT}); $z_i(\theta_1)$ for $i=2,\hdots,n+1$ are determined as the solution of a system of linear equations; this system is given in (\ref{eq:System1}). The column vectors ${\bf h}_k(\cdot)$ for $k=1,\hdots,n+1$ are defined in (\ref{vectors}).
\end{theorem}

The outer sum in (\ref{jointtransform}) ranging from 2 to $n+1$ represents the conditioning on one of the $n+1$ phases of the modulating Markov process when an upcrossing occurs, that is the event $\{\J(\sigma) = i\}$, $i=2,\hdots,n+1$. Observe that an upcrossing of level $\tau$ is not possible when $\J(\cdot)$ is in state 1 because then the process $X(\cdot)$ decreases. The terms $z_i(\theta_1)$, as defined in (\ref{eq:zi}), denote the transforms of $\sigma$ on the event the upcrossing of level $\tau$ occurs while the modulating Markov process is in state $i$, and the inner sum in (\ref{jointtransform}) concerns the transforms of $T$ conditional on the event $\{\J(\sigma)=i\}$. The Markov property of the workload process yields that 
\begin{equation}
\label{wi2}
\E\left[e^{-\theta_2 T} | E_i\right] = \E\left[e^{-\theta_2 T} | \J(\sigma)=i\right] =\E\left[e^{-\theta_2 T} | \J(0)=i\right] = w_i(\theta_2).
\end{equation}
\begin{proof}[Proof of Theorem \ref{th:Jointtransform}]The proof relies on the decomposition given in (\ref{eq: DEC}). Below we analyze the two expectations at the RHS of (\ref{eq: DEC}) separately.

$\circ$ We determine $z_i(\theta_1)$, for $i=2,\hdots,n+1$, as the solution of a system of linear equations; this idea was initially developed in \cite[Section 5]{Asm2} and essentially relies on the Kella-Whitt martingale for Markov Additive Processes. The Kella-Whitt martingale for a Markov Additive Process reflected at the infimum, has, for all $z\geq0$, $\theta_1\geq0$ and for $t\geq0$, the following form 
\begin{equation}
\label{eq:KellaWhitt}
{\bf M}(z,t) = \int_0^t e^{-z Q(s)-\theta_1s}{\bf e}_{\J(s)}{\rm d}s (\Q-z{\bf \Delta}_r-\theta_1\textbf{I})+e^{-z \tau}{\bf e}_{\J(0)} - e^{-z Q(t)-\theta_1t}{\bf e}_{\J(t)}-z\int_0^t e^{-\theta_1s}{\bf e}_{\J(s)}{\rm d}L(s).
\end{equation}
The expression above follows from the general form of the Kella-Whitt martingale given in (\ref{KWM}) by considering the process $Y(\cdot)$ defined by $Y(t):= \tau+L(t)+\theta_1 t/z$, for $t\geq0$. This gives $Z(t) = \tau+X(t)+L(t)+\theta_1 t/z = Q(t) + \theta_1 t/z$. The process $Y(\cdot)$ has paths of bounded variation and is also continuous since the local time at the infimum $L(\cdot)$ is a continuous process. Hence, $M(z,\cdot)$ is a zero-mean martingale. Furthermore, due to the construction of the model we have that $\J(0)=1$. Applying the optional sampling theorem for the stopping time $\sigma$, we obtain, for all $z\geq0$, 
\begin{equation}
\label{eq: OSA}
\E\left[\int_0^{\sigma} e^{-z Q(s)-\theta_1s}{\bf e}_{\J(s)}{\rm d}s\right](\Q-z{\bf \Delta}_r - \theta_1{\bf I}) = e^{-z\tau}{\bf z}(\theta_1) - e^{-z\tau}{\bf e}_1 + z{\bf \ell}(\theta_1),
\end{equation}
where
\[{\bf z}(\theta_1)=\E \left[e^{-\theta_1\sigma}{\bf e}_{\J(\sigma)}\right] = \Big(0,\E\left[e^{-\theta_1\sigma}1_{\{\J(\sigma)=2\}}\right],\hdots,\E\left[e^{-\theta_1\sigma}1_{\{\J(\sigma)=n+1\}}\right]\Big)=\Big(0,z_2(\theta_1),\hdots,z_{n+1}(\theta_2)\Big)\]
and
\[{\bm{\ell}}(\theta_1)=\E\left[\int_0^{\sigma} e^{-\theta_1s}{\bf e}_{\J(s)}{\rm d}L(s)\right]=\Big(\E\left[\int_0^{\sigma} e^{-\theta_1s}1_{\{\J(s)=1\}}{\rm d}L(s)\right],0,\hdots,0\Big) = \Big(\ell(\theta_1),0,\hdots,0\Big).\]
The row vector ${\bm\ell}(\theta_1)$ represents the local time at the infimum up to the stopping time $\sigma$; the process $Q(\cdot)$ can hit level 0 only on the event $\{\J(s) = 1\}$. Consider the $n+1$ roots of the equation $\text{det}\left(\Q-z{\bf \Delta}_r - \theta_1{\bf I}\right) = 0$, denoted by $\rho_1(\theta_1),\hdots,\rho_{n+1}(\theta_1)$, and the corresponding column vectors ${\bf h}_k(\theta_1)$, for $k=1,\hdots,n+1$ as defined in (\ref{vectors}). Substituting $z=\rho_k(\theta_1)$ in (\ref{eq: OSA}) and taking the inner products with the column vectors ${\bf h}_k(\theta)$ we obtain, for $k=1,\hdots,n+1$, the system of equations
\begin{equation}
\label{eq:System1}
{\bf z}(\theta_1)\cdot {\bf h}_k(\theta_1)  + e^{\rho_k(\theta_1)\tau}\rho_k(\theta_1)\ell(\theta_1) =1.
\end{equation}
Solving this system of equations we obtain the $z_i(\theta_1)$, for $i=2,\hdots,n+1$.

$\circ$ Next, consider the second expectation in each of the summands at the RHS of (\ref{eq: DEC}), i.e., the term $w_i(\theta_2)=\E\left[e^{-\theta_2T}| \J(0)=i\right]$. This expectation represents the transform of the first time the process $X(\cdot)$ hits level $\tau$ given that $\J(0) =i$, for $i=2,\hdots,n+1$. The Kella-Whitt martingale, for a MAFP reflected at $K$, has, for all $z\geq0, \theta_2\geq0$ and for $t\geq0$, the following form:
\begin{align*}
{\bf M}_K(z,t) &=\hspace{-3pt} \int_0^t\hspace{-3pt} e^{-zQ(s)-\theta_2s}{\bf e}_{\J(s)}{\rm d}s (\Q-z{\bf \Delta}_r-\theta_2\textbf{I})+e^{-z \tau}{\bf e}_{\J(0)} - e^{-z Q(t)-\theta_2t}{\bf e}_{\J(t)}\\
&+ze^{-z K}\hspace{-4pt}\int_0^t e^{-\theta_2s}{\bf e}_{\J(s)}{\rm d}\bar{L}(s) \numberthis\label{eq:KLWMAFP}.
\end{align*}
The expression above follows from the general form of the Kella-Whitt martingale given in (\ref{KWM}) by considering the process $Y(\cdot)$ defined by $Y(t):= \tau-\bar{L}(t)+\theta_1 t/z$, for $t\geq0$. This gives $Z(t) = \tau+X(t)-\bar{L}(t)+\theta_1 t/z = Q(t) + \theta_1 t/z$. A similar argument as for the stopping time $\sigma$ and (\ref{wi2}) yields $n$ systems of linear equations; for each $i=2,\hdots,n+1$, we solve for the unknowns $w_i(\cdot)$ and $\bar{\ell}_j(\cdot)$, $j=2,\ldots,n+1$, using the following system:
\begin{equation}
\label{System2}
 w_i(\theta_2) +\sum_{j=2}^{n+1}\bar{\ell}_j(\theta_2)\Big(e^{-\rho_k(\theta_2)(K-\tau)}\rho_k(\theta_2){\rm h}_{k,j}(\theta_2)\Big)={\rm h}_{k,i}(\theta_2) \hspace{5mm} \text{for} \hspace{2mm} k=1,\hdots,n+1,
\end{equation}
where $\bar{\ell}_j(\theta_2) = \E \left[\int_0^T e^{-\theta_2 s}1_{\{\J(s)=j\}}{\rm d}\bar{L}(s)\right], j=2,\hdots,n+1$. Using the method of determinants we can write $w_i(\theta_2)$ in the following form: 
\begin{equation}
\label{eq: Solutionw}
w_i(\theta_2) = \E\left[e^{-\theta_2 T}|\J(0)=i\right] = \frac{\sum_{k=1}^{n+1}(-1)^{1+k}c_k(\theta_2){\rm h}_{k,i}(\theta_2)}{\sum_{k=1}^{n+1}(-1)^{1+k}c_k(\theta_2)},
\end{equation}
where, for $k=1,\hdots,n+1$, 

\begin{equation}
\label{DET}
 c_k(\theta_2) = 
 \left|
\begin{array}{llll} 
\hspace{3mm}\rho_1(\theta_2){\rm h}_{1,2}(\theta_2)e^{-\rho_1(\theta_2)(K-\tau)}  &\hspace{3mm} \hdots & \hspace{3mm} \rho_1(\theta_2){\rm h}_{1,n+1}(\theta_2)e^{-\rho_1(\theta_2)(K-\tau)} \\ 
\hspace{15mm}\vdots & \hspace{5mm} \vdots &\hspace{15mm}\vdots\\
\rho_{k-1}(\theta_2){\rm h}_{k-1,2}(\theta_2)e^{-\rho_{k-1}(\theta_2)(K-\tau)}  &\hspace{3mm} \hdots & \rho_{k-1}(\theta_2){\rm h}_{k-1,n+1}(\theta_2)e^{-\rho_{k-1}(\theta_2)(K-\tau)}\\
\rho_{k+1}(\theta_2){\rm h}_{k+1,2}(\theta_2)e^{-\rho_{k+1}(\theta_2)(K-\tau)}  &\hspace{3mm} \hdots & \rho_{k+1}(\theta_2){\rm h}_{k+1,n+1}(\theta_2)e^{-\rho_{k+1}(\theta_2)(K-\tau)}\\
\hspace{15mm}\vdots & \hspace{5mm} \vdots  &\hspace{15mm}\vdots\\
\rho_{n+1}(\theta_2){\rm h}_{n+1,2}(\theta_2)e^{-\rho_{n+1}(\theta_2)(K-\tau)}  &\hspace{3mm} \hdots & \rho_{n+1}(\theta_2){\rm h}_{n+1,n+1}(\theta_2)e^{-\rho_{n+1}(\theta_2)(K-\tau)}\\
 \end{array} \right|
\end{equation}
Denoting 
\begin{equation}
\label{eq:CT}
C(\theta_2)=\sum_{k=1}^{n+1}(-1)^{1+k}c_k(\theta_2)
\end{equation}
and substituting the expression found for $w_i(\theta_2)$ in (\ref{eq: Solutionw}) into (\ref{eq: DEC}) yields the result of Theorem \ref{th:Jointtransform} with the terms $z_i(\theta_2)$, $i=2,\hdots,n+1$, given by the system of equations in (\ref{eq:System1}).
\end{proof}

\subsection{The finite buffer queue}\label{finite buffer queue}
Using the result of Theorem \ref{th:Jointtransform} we can also study the occupation time of the workload process in a {\em finite-buffer queue} with phase-type service time distribution. Consider a queue where customers arrive according to a Poisson process with rate $\lambda$ and have a phase-type service time distribution with representation $(n,{\bm\alpha}_0,{\bf T})$. Moreover, the queue has finite capacity $K$ and work is served with rate $r_1$. The workload process $\{Q(t)\}_{t\geq0}$ is modeled using a reflected compound Poisson process with negative drift $r_1<0$ and upward jumps with a phase-type $(n,{\bm\alpha}_0,{\bf T})$ distribution. Such a process has Laplace exponent equal to
\begin{equation}
\label{eq: Laplace}
\phi(s) = -s r_1-\lambda+\lambda\hat{B}[s]= -s r_1-\lambda+\lambda {\bm \alpha}_0^{\hspace{2pt} {\rm T}}(s\textbf{I}-\textbf{T})^{-1}{\bf t}, \hspace{4mm} s\geq0
\end{equation}
where ${\bf t} = -{\bf T 1}$.  As for the MAFP in Section \ref{sec: MAFP}
we determine the joint transform of the random variables $D$ and $U$. First we introduce some notation. Define,  for $k=1,\hdots,n+1$, the vectors ${\bf h}_k(\cdot)$ as:
\begin{equation}
\label{vectors 2}
{\rm h}_{k,1}(\cdot) = 1 \hspace{2mm} \forall k=1,\hdots,n+1 \hspace{2mm} \text{and} \hspace{2mm} {\rm h}_{k,j}(\cdot)=\hat{B}_j[p_k(\cdot)] \hspace{2mm} j=2,\hdots,n+1,
\end{equation}
where $p_k(q), k=1,\hdots,n+1$ are the $n+1$ roots of the equation $\phi(s) = q$. Consider the following system of linear equations
\begin{equation}
\label{eq:System23}
\sum_{j=2}^{n+1}z_j(\theta_1){\rm h}_{k,j}(\theta_1)+p_k(\theta_1)e^{p_k(\theta_1)\tau}\ell(\theta_1) = 1, \hspace{2mm} k=1,\hdots,n+1,
\end{equation}
and define $c_k(\cdot), k=1,\hdots,n+1$ and $C(\cdot)$ as in (\ref{DET}) and (\ref{eq:CT}) above with the defference that $\rho_k(\cdot)$ is replaced by $p_k(\cdot)$. 

\begin{corollary}
\label{cor: FBQ}
Consider a compound Poisson process with negative drift $r_1<0$ and upward jumps with a phase-type $(n,{\bm\alpha}_0,{\bf T})$ distribution. Consider the process reflected at the infimum and at level $K>0$. For $\theta_1,\theta_2\geq0$, the joint transform of the random variables $D$ and $U$ is given by 
\begin{equation}
\label{finite buffer queue result}
\E e^{-\theta_1 D-\theta_2 U} = \frac{1}{C(\theta_2)}\sum_{i=2}^{n+1}z_i(\theta_1)\sum_{j=1}^{n+1}(-1)^{j+1}c_j(\theta_2),
\end{equation}
where $c_j(\theta_2), j=1,\hdots,n+1$ and $C(\theta_2)$ are as above; $z_i(\theta_1)$ for $i=2,\hdots,n+1$ are determined as the solution of (\ref{eq:System23}) and ${\bf h}_k(\cdot), k=1,\hdots,n+1$ is as in (\ref{vectors 2}).
\end{corollary}
 The workload process $\{Q(t)\}_{t\geq0}$ can be studied as the limit of a MAFP in the following sense, see also \cite[Section 7]{AsmK}. Following the construction presented in Section \ref{subs:double-refl} we define, for $r>0$, the MAFP $\{X^r(t), \mathcal{J}^r(t)\}_{t\geq0}$ where the Markov process has state space $E=\{1,\hdots,n+1\}$ and generator $\mathcal{Q}^r$ given by 
\begin{equation*}
\Q^r = \begin{bmatrix}
       -\lambda &  \lambda\bf{\alpha}_0^{\rm T}\\[0.5em]
       r{\bf t}            & r{\bf T}
  
     \end{bmatrix},
 \end{equation*}
which is a $(n+1)\times(n+1)$ matrix. We also let the positive rates be equal, i.e, $r_2=\hdots=r_{n+1}=r$ and we send $r\rightarrow\infty$ later on. The assumptions on $\lambda, {\bf t}, {\bm \alpha}_0$ and \textbf{T} are the same as in Section \ref{subs:double-refl}. On the event $\{\mathcal{J}^r(\cdot)=1\}$ the process $X^r(\cdot)$ decreases with rate $r_1<0$ and on the event $\{\mathcal{J}^r(\cdot)=i\}$, for $i=2,\hdots,n+1$, the process $X^r(\cdot)$ increases with rate $r>0$. Such a MAFP decreases linearly with rate $r_1$ during OFF-times, which are exponentially distributed with parameter $\lambda$ and increases linearly with rate $r$ during  ON-times, which have a phase-type $(n,{\bm \alpha}_0,r\textbf{T})$ distribution. By multiplying the matrix ${\bf T}$ with the factor $r$ we see that the resulting phase-type distribution behaves like a phase-type distribution with representation $(n,{\bm\alpha}_0,{\bf T})$ divided by $r$. Using the representation in (\ref{eq:MAPl}) we see that letting $r\rightarrow \infty$ the process $(X^r(t),\mathcal{J}^r(t))_{t\geq0}$ converges path-wise to a compound Poisson process with linear rate $r_1<0$ and jumps in the upward direction with phase-type $(n,{\bm\alpha}_0,\textbf{T})$ distribution. The workload process $\{Q^r(t)\}_{t\geq0}$ converges to $\{Q(t)\}_{t\geq0}$, i.e. a reflected compound Poisson process, which follows by the continuity of the reflection operators with respect to the $D_1$ topology. Hence the joint transform of $D$ and $U$ is computed by using the result established in Theorem \ref{th:Jointtransform} and letting $r\rightarrow \infty$.

\section{Numerical Computation}
\label{sec: numerics}

In this note we have studied the occupation time of the workload process $\{Q(t)\}_{t\geq0}$ of the set $[0,\tau]$ upto time $t$ for the finite buffer fluid queue with a single state (of an independently evolving Markov chain) in which the workload decreases. Special cases are ON-OFF sources where the OFF times are exponential and the ON-times follow a phase type distribution; the M/G/1 queue with phase-type jumps can be studied as a limiting special case. By considering the process $\{K-Q(t)\}_{t\geq0}$, the results for fluid queues with a single state where work accumulates is now immediate; the same then holds for ON-OFF sources with phase-type OFF times and exponential ON times and doubly reflected risk reserve processes with negative phase-type jumps.

Essential in our analysis was the joint transform of the consecutive periods below and above $\tau$, i.e., $\E e^{-\theta_1 D - \theta_2 U}$ for $\theta_1\geq0,\theta_2\geq0$. The double transform of the occupation time uniquely specifies its distribution, which can be evaluated by numerically inverting the double transform \cite{Aba 2}. Such a procedure has been carried out in \cite{Roubos} for the M/M/$s$ queue, where an explicit expression for the double transform can be derived. For the current model, the transform is given implicitly, where for given $\theta_1, \theta_2$ linear equations need to be solved. The methodology we present is rather straightforward to implement and yields an approximation for the distribution function of the occupation time up to machine precision. We first present an algorithm which uses Theorem \ref{jointtransform} in order to numerically approximate the distribution function (or the density function) of the occupation time $\alpha(\cdot)$. Afterwards using the technique presented above we let $r\rightarrow \infty$ and obtain the distribution function of the occupation time $\alpha(\cdot)$ in the M/G/1 queue with phase-type service distribution. 
\paragraph{Algorithm:} {\textbf{Input:}}~ $t\geq0,s\in[0,\infty)$ ~ {\textbf{Output:}} The distribution function $\Pb(\alpha(t)\leq s)$.  
\begin{itemize} 
\item[(1)] Compute $L_{1,2}(\theta_1,\theta_2)=\E e^{-\theta_1D-\theta_2U}$ using Theorem \ref{th:Jointtransform} and by solving the systems of linear equations given in 
(\ref{eq:System1}) and (\ref{System2}). 
\item[(2)] Compute $L_1(\theta)$ using Theorem \ref{th:Jointtransform} and setting $\theta_2=0$.
\item[(3)] Compute the double transform of the occupation time $\alpha(\cdot)$ using Theorem \ref{Old}.
\item[(4)] Use Laplace inversion in order to compute $\Pb(\alpha(t)\leq s)$.
\end{itemize}
In this Section we numerically compute the right hand sides of (\ref{jointtransform}) and we obtain a numerical approximation up to machine precision of the joint transform $\E e^{-\theta_1 D-\theta_2 U}, \theta_1\geq0, \theta_2\geq0$. We then use (\ref{DoubleTransform}) and the Laplace inversion techniques presented in \cite{Aba 2} to compute the distribution function of the occupation time. We used the Euler-Euler algorithm with $M=10$. In Figure \ref{fig: distribution function} below we present the density function of the occupation time $\alpha(\cdot)$ for a MAFP (Left) and the M/G/1 queue (Right). For both cases we consider a time horizon of $t=100$ time units and an arrival rate equal to $\lambda=1.05$. The parameters of the two models are chosen in such a way so that the load in the system is the same for all cases and equal to $\rho = 0.945$. The levels $\tau$ and $K$ are chosen equal to $\tau=0.8$ and $K=2$. We consider MAFPs with ON- times having an Erlang, exponential and Coxian distribution which have coefficient of variation less than one, equal to one and greater than one.  For the exponential distribution we choose $\mu=2$, for the Erlang we choose $m=2$ (2 phases) and $\mu=6$ and for the Coxian we choose $m=2$ (2 phases), $p=0.5,\mu_1=  18$ and $\mu_2=2.25$. For all three cases we choose a depletion rate $r_1=-1$ and the system size increases with rates $r_2=1.8$ and $r_3=3.6$. For the M/G/1 queue we consider the cases the jump size has an Erlang, exponential and Coxian distribution. For the exponential distribution we choose $\mu=1.111$, for the Erlang we choose $m=2,4$ (2 and 4 phases) and $\mu=2.222, 4.444$ and for the Coxian we choose $m=2$ (2 phases), $p=0.5$, $\mu_1= 5.555$ and $\mu_2=0.694$. We treat the case of an Erlang(4) jump distrbution only for the M/G/1 queue but similar results can be derived for the fluid process as well.

\begin{figure}[H]
\label{fig: distribution function}
\centering
\includegraphics[scale=0.5]{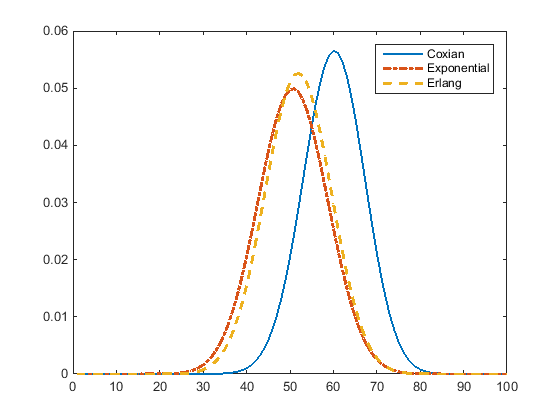}
\includegraphics[scale=0.5]{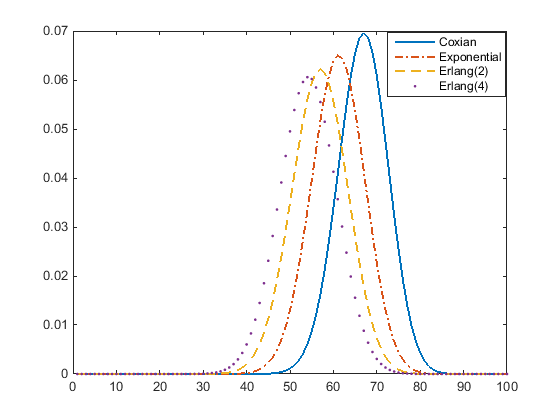}
\caption{Distribution function and density of the occupation time}
\end{figure}
In the following two figures we vary the buffer capacity $K$; the process we considered was the MAFP with OFF-times having a Coxian distribution. The parameters of the distribution are the same as above for the MAFP.  

\begin{figure}[H]
\label{fig: buffer content varies}
\includegraphics[scale=0.5]{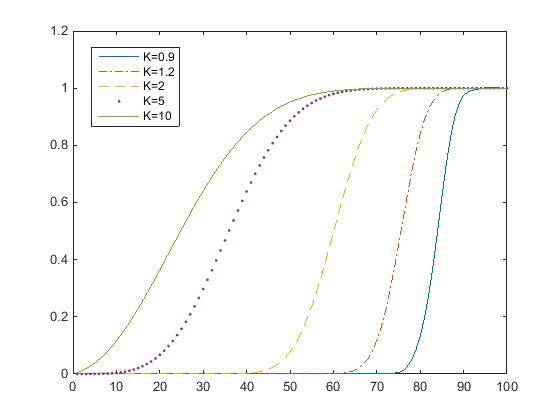}
\includegraphics[scale=0.5]{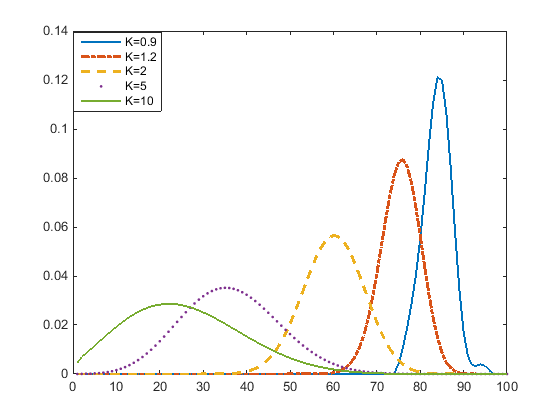}
\caption{Distribution function and density of the occupation time as $K$ varies}
\end{figure}

In the figure below we consider the M/G/1 queue (left figure) and we vary the arrival rate in order to study the occupation time for different values of the occupation rate $\rho$. The service time has an Erlang distribution with two phases and $\mu=2.222$. In the second figure we study the convergence of the distribution function of the occupation time for the case all the positive rates are set equal and take the limit as $r\rightarrow\infty$, see also Section \ref{finite buffer queue}. In this case we consider a MAFP with Coxian ON-times with parameters as above. 

 \begin{figure}[H]
\label{fig: last two}
\centering
\includegraphics[scale=0.5]{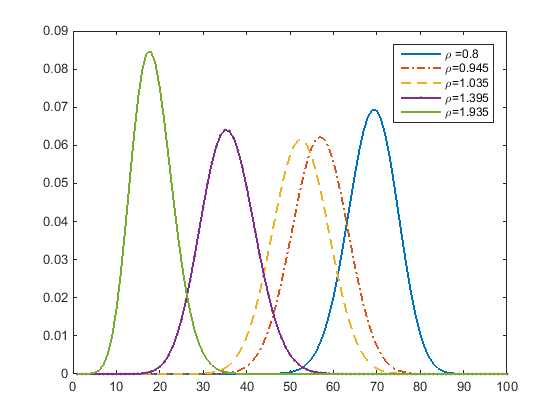}
\includegraphics[scale=0.5]{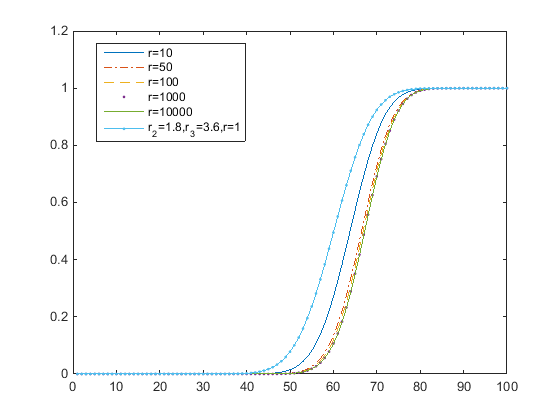}
\caption{Occupation time as occupation rate varies (Left) and the occupation time as $r$ grows (Right)}
\end{figure}

\section*{Acknowledgements}

We would like to thank the associate editor for his inspiring comments and the anonymous referee who analysed the paper with great care. Their suggestions helped us 
improve the paper significantly.

The research of N. Starreveld and M. Mandjes is partly funded by the NWO Gravitation project
N{\sc etworks}, grant number 024.002.003.


\end{document}